\documentclass[11pt]{article}
\usepackage{amsmath,amssymb,amsfonts,epsfig,enumerate,xcolor,amsthm} 
\usepackage{hyperref}    
\usepackage{graphicx}  
\usepackage[capitalise]{cleveref}
\usepackage{bbold}
\usepackage[normalem]{ulem}
\usepackage[numbers,sort&compress]{natbib}
\hypersetup{colorlinks=true,linkcolor=green,pdfborderstyle={/S/U/W 1}}

\newcommand{\floor}[1]{{\lfloor #1 \rfloor}}

\newcommand{\Z}{\mathbb{Z}}
\newcommand{\ccF}{\mathcal{F}}
\newcommand{\ccT}{\mathcal{T}}
\usepackage{mathrsfs}  
\newcommand{\ssF}{\mathscr{F}}
\newcommand{\ssT}{\mathscr{T}}

\theoremstyle{plain}
\newtheorem{theorem}  {Theorem}  [section]
\newtheorem{lemma}  [theorem]   {Lemma}

\newtheorem{fact} [theorem] {Fact}
\newtheorem{proposition} [theorem] {Proposition}
\newtheorem{claim} [theorem] {Claim}
\newtheorem{conjecture}[theorem] {Conjecture}
\newtheorem{problem}[theorem] {Problem}
\theoremstyle{definition}
\newtheorem{definition}[theorem] {Definition}
\newtheorem{example}[theorem] {Example}
\newcommand{\claimproof}{\renewcommand{\qedsymbol}{$\diamond$}}

\newcommand\psum{\mathbin{\raisebox{0.105ex}{%
  \rotatebox[origin=c]{180}{$\uplus$}}}}
\DeclareMathOperator{\setpsum}{\psum}
\DeclareMathOperator{\setssum}{+}
 
\newcommand{\bfs}[1]{{\mathbf{s}(#1)}}
\newcommand{\bfsb}[1]{{\bar{\mathbf{s}}(#1)}}
\newcommand{\SSs}[1]{\bfs{#1}}
\newcommand{\SSb}[1]{\bfsb{#1}}
\newcommand{\Sz}{\SSs{0}}
\newcommand{\Szb}{\SSb{0}}
\newcommand{\So}{\SSs{1}}
\newcommand{\Sob}{\SSb{1}}
\newcommand{\St}{\SSs{2}}
\newcommand{\Stb}{\SSb{2}}
\newcommand{\Sth}{\SSs{3}}

\newcommand{\ccFF}[1]{\ccF_{\SSs{#1}}}
\newcommand{\ccFFb}[1]{\ccF_{\SSb{#1}}}
\newcommand{\Fz}{\ccFF{0}}
\newcommand{\Fzb}{\ccFFb{0}}
\newcommand{\Fo}{\ccFF{1}}
\newcommand{\Fob}{\ccFFb{1}}
\newcommand{\Ft}{\ccFF{2}}
\newcommand{\Ftb}{\ccFFb{2}}

\newcommand{\ccFT}[1]{\ccT_{\SSs{#1}}}
\newcommand{\ccFTb}[1]{\ccT_{\SSb{#1}}}

\newcommand{\Tz}{\ccFT{0}}
\newcommand{\Tzb}{\ccFTb{0}}
\newcommand{\To}{\ccFT{1}}
\newcommand{\Tob}{\ccFTb{1}}
\newcommand{\Tt}{\ccFT{2}}
\newcommand{\Ttb}{\ccFTb{2}}

\newcommand{\us}{s} % was u 
\newcommand{\vt}{t} % was v
 
\newcommand{\FS}{\ccF_S}

\newcommand{\FT}{\ccF_T}

\begin{document}

\title{Obstructions for homomorphisms to odd cycles in series-parallel graphs}

 \author{Eun-Kyung Cho\thanks{
 Department of Mathematics, Hanyang University, Seoul, Republic of Korea.
  \texttt{ekcho2020@gmail.com}
 }
 \and Ilkyoo Choi\thanks{
 Department of Mathematics, Hankuk University of Foreign Studies, Yongin-si, Gyeonggi-do, Republic of Korea.
 \texttt{ilkyoo@hufs.ac.kr} and  Discrete Mathematics Group, Institute for Basic Science (IBS), Daejeon, Republic of Korea. 
 }
 \and Boram Park\thanks{
 Department of Mathematics, Ajou University, Suwon-si, Gyeonggi-do, Republic of Korea.
 \texttt{borampark@ajou.ac.kr}
 } 
 \and Mark Siggers\thanks{Department of Mathematics, Kyungpook National University, Daegu, Republic of Korea.
 \texttt{mhsiggers@knu.ac.kr} }}
\maketitle

%\smallskip \noindent \textbf{Subject class.} [2020]{05C15,05C10} 
% 05C10: planar graphs
% 05C15: graph colouring
% 05C21: flows in graphs 

\begin{abstract}
For a graph $H$, an {\em $H$-colouring} of a graph $G$ 
is a vertex map $\phi:V(G) \to V(H)$ such that adjacent vertices are mapped to adjacent vertices.  
A graph $G$ is {\em $C_{2k+1}$-critical} if $G$ has no $C_{2k+1}$-colouring but every proper subgraph of $G$ has a $C_{2k+1}$-colouring. 
We prove a structural characterisation of $C_{2k+1}$-critical graphs when $k  \geq 2$.
In the case that $k = 2$, we use the aforementioned charazterisation to show 
 a $C_3$-free series-parallel graph $G$ has a $C_5$-colouring if either $G$ has neither $C_8$ nor $C_{10}$, or $G$ has no two $5$-cycles sharing a vertex. 
\end{abstract}


\section{Introduction}

Following Gr\"otzsch's Theorem~\cite{Grotzsch}, which tells us that planar graphs with no $C_3$ have a proper $3$-colouring, Havel~\cite{Havel} conjectured the stronger statement that
any planar graph whose $C_3$s are no closer than some constant $c$ also admit a proper $3$-colouring.  This was proved by Dvo\v{r}\'ak, Kr\'{a}\v{l},
and Thomas in~\cite{DKT} with $c \geq 10^{100}$. 
Steinberg~\cite{Steinberg} conjectured that any planar graph with neither a $4$-cycle nor a $5$-cycle is also properly $3$-colourable, but a counterexample was found by Cohen-Addad et al.~\cite{SteinFalse}. 
In this paper we prove Havel- and Steinberg-type statements for $C_5$-colourings of series-parallel graphs. 

Given two graphs $G$ and $H$, 
a {\em homomorphism} from $G$ to $H$ is a vertex map $\phi:V(G) \to V(H)$ such that adjacent vertices are mapped to adjacent vertices.  A homomorphism of $G$ to $H$ is also called an {\em $H$-colouring} of $G$, as then a $K_k$-colouring of $G$ is exactly a proper $k$-colouring of $G$.

As a proper $3$-colouring is exactly a $C_3$-colouring, and larger odd cycles admit homomorphisms to smaller odd cycles, as $k \geq 2$ increases $C_{2k+1}$-colourings can be viewed as increasingly finer refinements of proper $3$-colourings.

 Jaeger's Modular Orientation Conjecture~\cite{Jaeger84}, when restricted to planar graphs, and stated for the dual graph, posits the following generalisation, the $k = 1$ case of which is Gr\"otzch's Theorem.

\begin{conjecture}\label{conj:Jaeger-dual}
For a positive integer $k$, every planar graph of girth at least $4k$ has a $C_{2k+1}$-colouring. 
\end{conjecture}

Let $\zeta(2k+1)$ be the minimum girth for which all planar graphs of girth at least $\zeta(2k+1)$ have a $C_{2k+1}$-colouring. 
It is easily shown, for $k \geq 1$, that $4k \leq \zeta(2k+1)$, and the current best general upper bound is by Lov\'asz et al.~\cite{LTWZ13} where they show that $\zeta(2k+1) \leq 6k+1$. For small $k$ there are even better upper bounds.

In the particular case of $k=2$, \Cref{conj:Jaeger-dual} says that $\zeta(5) = 8$. 
While we have $8 \leq \zeta(5)$, Dvo\v{r}\'ak and Postle~\cite{DP17} proved a result that implies $\zeta(5) \leq 10$.

Pan and Zhu found the gap easier to close on for {\em series-parallel graphs}; this well-known subclass of planar graphs is defined in the next section.  
Let $\zeta_{sp}(2k+1)$ be the minimum girth for which all series-parallel graphs of girth at least $\zeta_{sp}(2k+1)$ have a $C_{2k+1}$-colouring. 
Proving the upper bound in~\cite{PanZhu1}, Pan and Zhu then showed in~\cite{PanZhu2}, for odd $g \geq 5$,  that 
\begin{equation}\label{eq:pz}
    \zeta_{sp}(g) =  g + 2\floor{(g-1)/4}.
\end{equation}
They did this by constructing series-parallel graphs of girth $\zeta_{sp}(g)-1$ with no $C_g$-colourings.  For this task Pan and Zhu used what they called the `labelling method'. 
The idea was to construct graphs via the two operations (which we recall below) from which all series-parallel graphs are known to be constructible, and to keep track, with labels, of the possible $C_g$-colourings of the graph.
The results of Pan and Zhu that are mentioned above are more general, dealing not only with $C_g$-colourings, but circular $(p,q)$-colourings.

When deciding if a graph has a property such as $C_{2k+1}$-colourability, it is of course useful to know the minimal set of obstructions to the graph property.  Following~\cite{DP17}, we say a graph is {\em $C_{2k+1}$-critical} if it does not admit a $C_{2k+1}$-colouring but every proper subgraph has a $C_{2k+1}$-colouring.

Using the labelling method and keeping track not only which colourings a graph has as we construct it, but also  which of these colourings are `critical', 
we are able to characterise, and methodically construct without the heavy computational task of checking criticality, the $C_5$-critical series-parallel graphs in \Cref{thm:main}.  
The characterisation is somewhat technical, but yields some nice consequences.  We use it to give Havel- and Steinberg-type strengthenings of Pan and Zhu's Jaeger-type result that series-parallel graphs of odd girth at least $7$ have a $C_5$-colouring.
In \Cref{cor:twoC5s} we get that a series-parallel graph is $C_5$-colourable if it has odd girth at least $5$ and every pair of $C_5$s are distance at least $1$ apart. 
We also get that a series-parallel graph is $C_5$-colourable if it has odd girth at least $5$ and contains neither $C_8$ nor $C_{10}$.  
The point of such results for series-parallel graphs is, of course, as a starting point for finding analogues for planar graphs.

 In \Cref{cor:minimal_2k+1}, we extend  the  technical characterisation, \Cref{thm:main}, of $C_5$-critical series-parallel  graphs to a characterisation of $C_{2k+1}$-critical series-parallel  graphs. We expect this could be useful in getting Havel-type versions of \eqref{eq:pz}.

\begin{problem}\label{P}
For odd $g \geq 7$, is there a constant $d_g$ such that when a graph $G$ has odd girth $\zeta_{sp}(C_{g}) - 2$ and there are no shortest odd cycles that are within distance $d_g$ of each other, then $G$ has a $C_g$-colouring?
\end{problem}

\section{Preliminaries}

In this section we lay out some preliminary facts that will be used in the remaining of the paper. 
One advantage that series-parallel graphs have as a subclass of planar graphs is that they can be defined recursively, as follows.  A {\em $2$-terminal graph} $(G,\us,\vt)$ is a  graph $G$ with two distinguished distinct vertices $\us$ and $\vt$, which we call {\em terminals}.  More succinctly, we sometimes call $G$ an {\em $(\us,\vt)$-terminal graph}, and sometimes omit explicitly mentioning the terminal vertices when there is no confusion.

Let $(G_1, \us_1, \vt_1)$ and $(G_2, \us_2, \vt_2)$ be two $2$-terminal graphs.
The {\em serial sum} $G_1 + G_2$ of $G_1$ and $G_2$ is the $2$-terminal graph $(G_1 + G_2, \us_1, \vt_2)$ obtained by identifying $\vt_1$ and $\us_2$.
 The {\em parallel sum} 
 $G_1\psum G_2$ of $G_1$ and $G_2$ is the $2$-terminal graph $(G_1 \psum G_2, \us,\vt)$ where $\us$ is the vertex obtained by identifying $\us_1$ and $\us_2$ and $\vt$ is the vertex obtained by identifying $\vt_1$ and $\vt_2$.

A graph $G$ is {\em series-parallel} if, for some choice of terminals $\us$ and $\vt$, $(G, \us,\vt)$ can be constructed from edges via the serial sum and parallel sum.  When we talk of the terminals of a series-parallel graph, it is one of these pairs of terminals.

 We will use the ring $\Z_{n}$ of integers modulo $n$ for the vertex set of $C_{n}$ making two vertices $u$ and $v$ adjacent if $|u - v| = 1$; thus a $C_{n}$-colouring of a graph $G$ is a
 mapping $\phi:V(G) \to \Z_{n}$ such that $|\phi(u)-\phi(v)|=1$ for every pair of adjacent vertices $u$ and $v$. 
 
For a nonempty  subset $S \subseteq \Z_{n}$, a $2$-terminal graph $(G,\us,\vt)$ is {\em $S$-forcing} (or $S$ is the {\em forced set} of $(G,\us,\vt)$) {with respect to $C_{n}$}, 
if 
     \[ S = \{ x \in \Z_{n}\mid \exists {\text{ 
     a homomorphism }} \phi: V(G) \to C_n \text{ s.t. } \phi(\us) = 0 \mbox{ and } \phi(\vt) = x\}.\]
For an element $x$ of the forced set of $(G,\us,\vt)$, we say $x$ is a forced element of $G$ or $G$ forces $x$.

When $n$ is odd, these forced sets have an obvious symmetry. 
A subset $S$ of the ring $\Z_{2k+1}  = \{0, \pm1, \dots, \pm k\}$ is {\em symmetric} if it is closed under multiplication by $-1$.  It is clear that every forced set of a $2$-terminal graph with respect to $C_{2k+1}$ is symmetric.
Moreover, it is easy to see how the forced sets of $2$-terminal graphs act under the serial sum and parallel sum.

Extending the sum operations on $2$-terminal graphs to families $\ccF_1$ and $\ccF_2$ of $2$-terminal graphs, let 
    \[ \ccF_1 \setssum \ccF_2 = \{ G_1 + G_2 \mid G_1 \in \ccF_1, G_2 \in \ccF_2 \}, \qquad \mbox{ and }\qquad        \ccF_1 \setpsum \ccF_2 = \{ G_1 \psum G_2 \mid G_1 \in \ccF_1, G_2 \in \ccF_2\}.  \]
At the same time, for two subsets $S_1$ and $S_2$ of $\Z_n$, let
      \[ S_1 + S_2 = \{ x + y \mid x \in S_1, y \in S_2\}. \]

It is clear the following holds by definition. 
\begin{quote}
$(\S)$  If $G_1$ is $S_1$-forcing and $G_2$ is $S_2$-forcing, then $G_1 + G_2$ is $(S_1 + S_2)$-forcing and $G_1 \psum G_2$  is $(S_1 \cap S_2)$-forcing. 
\end{quote}

 A $2$-terminal graph $(G,\us,\vt)$ that is $S$-forcing for a nonempty symmetric subset $S\subseteq\Z_n$ is {\em restricted} if $S\neq \Z_n$, otherwise it is {\em unrestricted}. It is {\em minimally $S$-forcing} if the forced set of $(G,\us,\vt)$ is $S$, but no proper subgraph $(G',\us,\vt)$ of $(G, \us, \vt)$ is $S$-forcing. 

For every $i\in \Z_{2k+1}$, we let $\bfs{i}=\{i,-i\}$ and 
 $\bfsb{i}=\Z_n\setminus\{i,-i\}$. 
 
 \begin{example}\label{ex:sub5}
The nonempty symmetric proper subsets of  $\Z_5$ are the following; their sums are shown in \Cref{table:C5}. 
\[ \begin{array}{lll}\notag
       \Sz= \{ 0 \} & \So =\bfs{4}= \{ -1, 1\} & \St =\bfs{3}= \{-2,2\} \\\notag
       \Szb= \{-2,-1,1,2\} & \Sob=\bfsb{4}= \{-2,0,2\} & \Stb=\bfsb{3}= \{-1,0,1\}.  
   \end{array}\]

\end{example}
\begin{table}%[h!]
\centering
\begin{tabular}{c||c|c|c|c|c|c}
     $+$ & $\Sz$ & $\So$ & $\St$ & $\Szb$ & $\Sob$ & $\Stb$  \\ \hline \hline
     $\Sz$ & $\Sz$ & $\So$ & $\St$ & $\Szb$ & $\Sob$ & $\Stb$  \\ \hline
     $\So$ & $\So$ & $\Sob$ & $\Szb$ & $\Z_5$ & $\Szb$ & $\Z_5$ \\ \hline
     $\St$ & $\St$ & $\Szb$ & $\Stb$ & $\Z_5$ & $\Z_5$ & $\Szb$ \\ \hline
     $\Szb$ & $\Szb$ & $\Z_5$ & $\Z_5$ & $\Z_5$ & $\Z_5$ & $\Z_5$ \\ \hline
     $\Sob$ & $\Sob$ & $\Szb$ & $\Z_5$ & $\Z_5$ & $\Z_5$ & $\Z_5$ \\ \hline
     $\Stb$ & $\Stb$ & $\Z_5$ & $\Szb$ & $\Z_5$ & $\Z_5$ & $\Z_5$ 
\end{tabular}
\vspace{0.3cm}\\
\caption{Sums of symmetric proper  subsets in $\Z_5$.}\label{table:C5} 
\end{table}

Though we will not need it, it is not too hard to show that for each symmetric subset of $\Z_5$ shown in \Cref{ex:sub5}, there is a series-parallel graph for which it is (minimally) forcing. Taking two that force disjoint sets, it is easy to then construct, with a parallel sum, graphs that are $C_5$-critical.  Our goal, however, is to describe the construction of all $C_5$-critical series-parallel graphs.  The following is useful for this.

\medskip

\begin{lemma}\label{lem:minimal_2k+1}
For a {positive integer $k$}  {and} a $C_{2k+1}$-critical series-parallel graph $G$, the following hold:
    \begin{itemize}
    \item[(i)] $G$ is $2$-connected.
    \item[(ii)] If $\{\us,\vt\}$ is a vertex cut of $G$ and  $H$ is a component of $G-\{\us,\vt\}$, then $G[V(H)\cup\{\us,\vt\}]$ is a restricted $(\us,\vt)$-terminal series-parallel graph.
\end{itemize}
\end{lemma}

\begin{proof}
 If $G$ is a complete graph $K_m$ for some $m$, then  {since $G$ is  $C_{2k+1}$-critical,} we have that $m \ge 3$ so that (i) holds, and (ii) holds as $G$ has no vertex cut. 
So we may assume that $G$ is not a complete graph.
Let $S$ be a minimum vertex cut of $G$, and let $H$ be a component of $G-S$. 
Let $J_1=G[V(H)\cup S]$ and $J_2=G-V(H)$.
By the minimality of $G$, $J_1$ and $J_2$ have $C_{2k+1}$-colourings $\phi_1$ and $\phi_2$, respectively. 

(i) Towards contradiction, suppose that $S=\{\vt\}$. 
We can permute the colours on $J_1$ so $\phi_1$ and $\phi_2$ agree on $\vt$. 
This gives a $C_{2k+1}$-colouring of $G$, which is a contradiction. 

(ii) Towards contradiction, suppose that $S=\{\us, \vt\}$, but that $(J_1, s,t)$ is unrestricted. 
We can permute the colours on $J_1$ so that $\phi_1$ and $\phi_2$ agree on $\us$ and $\vt$. 
This is a $C_{2k+1}$-colouring of $G$, which is a contradiction. 
\end{proof}

\section{Obstructions to $C_5$-colouring}

In this section, we completely characterize $C_5$-critical series-parallel graphs.
Recall that a $C_5$-critical graph is a graph with no $C_5$-colouring such that every proper subgraph has a $C_5$-colouring.

For a symmetric subset $S$ of $\Z_5$ let $\ccF_S$ be the family of all minimally $S$-forcing series-parallel graphs.  The first step is to characterise the families $\ccF_S$. We do so by proving that $\ccF_S$ is the following recursively defined family $\ccT_S$; the proof is presented in the next subsection.

\begin{definition}\label{def:families:ccF}
The families in $\ssT := \{\Tz, \Tzb, \To, \Tob, \Tt, \Ttb \}$ are the minimal families such that $K_2 \in \To$ and the following hold:
\begin{itemize}
\item[\rm(i)] For every distinct $i,j\in \{0,1,2\}$, if $G\in \ccFTb{i}$ and $H\in \ccFTb{j}$, then $G\psum H\in\ccFT{k}$, where $\{k\}=\{0,1,2\}\setminus\{i,j\}$.
    \item[\rm(ii)] If $G,H\in \ccFT{i}$ for some $i\in\{1,2\}$, then $G+H\in \ccFTb{i}$.
     \item[\rm(iii)] If $G\in \ccFT{i}$ and $H\in \ccFTb{i}$ for some $i\in\{1,2\}$, then $G+H\in \Tzb$.
     \item[\rm(iv)]  If $G \in \ccT_S$ and $H\in \Tz$ for some $\ccT_S \in \ssT$, then $G+H \in \ccT_S$.
\end{itemize} 
\end{definition}

\begin{theorem}\label{lem:families:ccF}
For each $S\in \{  \Sz,\So,\St, \Szb,\Sob,\Stb\}$, $\ccT_S$ is the family of minimally $S$-forcing  {$2$-}terminal series-parallel graphs $(G,\us,\vt)$, that is, $\ccT_S=\ccF_S$.
\end{theorem}

Let $\ssF := \{\Fz, \Fzb, \Fo, \Fob, \Ft, \Ftb \}$, and denote the union of all families in $\ssF$ by $\cup\ssF$.  
Observe that for every pair of distinct nonempty symmetric sets $S,T\in\{ \Sz,\So,\St, \Szb,\Sob,\Stb\}$,  if $A\in \FS$, $B\in \FT$, and {$A \psum B \in \cup \ssF$}, then 
\begin{eqnarray}\label{eq:ST}
&&S \cap T \neq T,   
\end{eqnarray}
as otherwise {$A \psum B$} would not be minimal. 

In addition, the recursive conditions (i)$\sim$(iv) in Definition~\ref{def:families:ccF} imply from Theorem~\ref{lem:families:ccF} that the families in $\ssF$ are the minimal families such that $K_2 \in \Fo$, $\ccF_S \supseteq \ccF_S \setssum \Fz$ for every symmetric set $S \in \{\Sz,\So, \St, \Szb, \Sob, \Stb\}$, and the following hold.
     \[ (*) \qquad \quad \begin{array}{lll}\notag
           \Fz \supseteq\Fob \setpsum \Ftb  & \hspace{1cm} & \Fzb \supseteq (\Fo \setssum \Fob)   \cup  (\Ft \setssum \Ftb ) \\ \notag
           \Fo \supseteq \Fzb \setpsum \Ftb && \Fob \supseteq \Fo \setssum \Fo   \\  \notag
          \Ft\supseteq \Fzb \setpsum \Fob   && \Ftb \supseteq \Ft \setssum \Ft
       \end{array}  \]
With this recursive definition of the families in $\cup \ssF$ one can easily compute small members of the families. 
\Cref{fig:smallgraphs} shows all graphs in $\cup \ssF$ with at most 10 vertices; we refer to these as the {\em base} graphs. We point out that $\Fo \setssum \Ft$ is excluded in $(*)$. 
As $\So + \St = \Szb$, one might expect $\Fo \setssum \Ft$ to be included in the identity for $\Fzb$. However, the serial sum of $G \in \Fo$ and $G' \in \Ft$, though $\Szb$-forcing, turns out not to be minimally forcing; one sees, for example, that $H_4 + H_1$ contains $H_3$ as a proper subgraph.

\begin{figure}[h!]
    \centering
    \includegraphics[width=14cm]{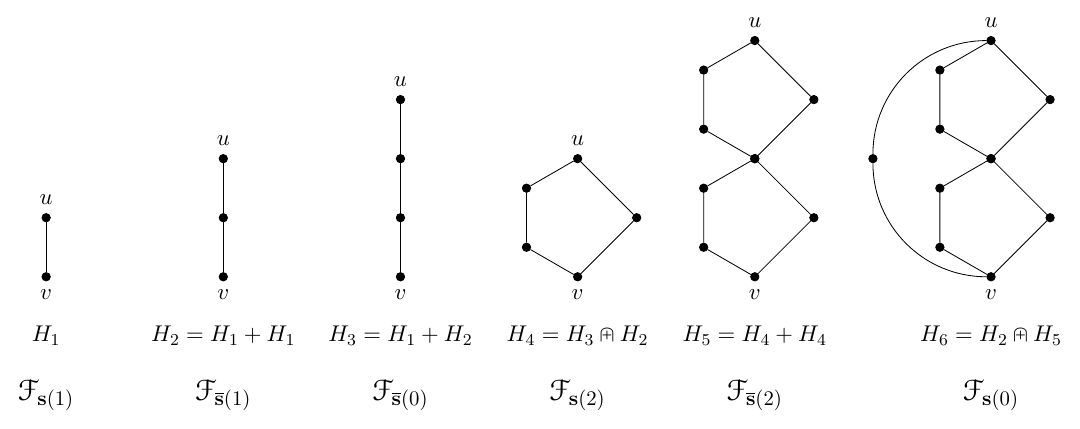}
    \caption{Base graphs -- graphs in $\cup \ssF$ with at most 10 vertices.}
    \label{fig:smallgraphs}
\end{figure}

The following is our main result. 

\begin{theorem}\label{thm:main}
For a series-parallel graph $G$,  $G$ is $C_5$-critical if and only if $G$ belongs to $\ccF_{\SSs{i}} \setpsum \ccF_{\SSb{i}}$ for some $i \in \{0,1,2\}$.
\end{theorem} 
By Theorems~\ref{lem:families:ccF} and~\ref{thm:main}, we can construct all $C_5$-critical graphs. 
In addition, we can show the following.
 
\begin{theorem}\label{cor:twoC5s}
Let $G$ be a series-parallel graph with odd girth at least $5$. Then $G$ is $C_5$-colourable if $G$ satisfies one of the following:
\begin{itemize}
    \item[(i)] $G$ has no two $5$-cycles sharing a vertex, or 
    \item[(ii)] $G$ has neither $8$-cycles nor $10$-cycles. 
\end{itemize}
\end{theorem}

\subsection{Proof of Theorem~\ref{lem:families:ccF}} 

As all graphs in $\ccT_S$ in \Cref{def:families:ccF} are constructed recursively from $K_2$ with the serial sum or parallel sum,  they are series-parallel graphs with designated terminal vertices.  For simplicity, let $\cup\ssT$ be the union of all families in $\ssT$.  
      
\begin{lemma}\label{lem:ccF-claim1}
Every graph in $\ccT_S$ is minimally $S$-forcing, for all $S \in \{ \bfs{0}, \bfs{1}, \bfs{2}, \bfsb{0}, \bfsb{1}, \bfsb{2}\}$.
\end{lemma} 
      
\begin{proof} 
Our proof is by induction on the number of edges of $G \in \cup\ssT$. It is clearly true when $G = K_2$. Take any graph $G \in \ccT_X$ for some proper symmetric subset $X$ of $\Z_5$.
By \Cref{def:families:ccF},  we may assume that either $G=G_1+G_2$ or $G=G_1\psum G_2$ where for $i\in\{1,2\}$ we may assume by the induction hypothesis that $G_i$ is a minimally $S_i$-forcing graph, so $G_i$ is in $\ccF_{S_i}$. 
By ($\S$), Table~\ref{table:C5}, and \Cref{def:families:ccF},   
it follows that $(G,\us,\vt)$ is $X$-forcing. What has to be shown is that it is minimally $X$-forcing, that is, every proper subgraph {$(H,\us,\vt)$} of {$(G,\us,\vt)$} is not $X$-forcing.

Let $e$ be any edge in $G$, and let $G'=G- e$.
We will show that $(G',\us,\vt)$ is not $X$-forcing.
Without loss of generality, we may assume that $e$ is an edge of $G_1$. 
As $G_1$ is minimally $S_1$-forcing, we have that $G_1': = G_1 - e$ is $S_1'$-forcing for some proper superset $S_1'$ of $S_1$.
If $S_1'=\Z_5$, then $S_1'+S_2=\Z_5$ and $S_1'\cap S_2=S_2$, and so either $G'$ is unrestricted or $G'$ is $S_2$-forcing, which means that $G'$ is not $X$-forcing
(since in this case, $X = S_1 \cap S_2 \neq S_2$ by \eqref{eq:ST}).
Suppose that $S_1'\subsetneq \Z_5$. 
By cardinality, $S_1=\bfs{i}$ for some $i\in \{0,1,2\}$. 
Referring to~\Cref{def:families:ccF}, this precludes the possibility that $G = G_1 \psum G_2$, so  $G=G_1+G_2$ and $X=S_1+S_2$. 
If $S_1=\Sz$, then $S_1' = \Sob$ or $S_1' = \Stb$. 
In either case, $S_1'+S_2$ is a proper superset of $X=S_1+S_2=S_2$, and  $G'$ is not $X$-forcing. 
So we may assume that $S_1\in\{\So, \St\}$. 
If $S_2= \Sz$, then $S_1'+S_2=S_1'$ is a proper superset of $X=S_1+S_2=S_1$, and $G'$ is not $X$-forcing.
Thus, by~\Cref{def:families:ccF}, $S_2 = \bfs{i}$ or $S_2= \bfsb{i}$.
Without loss of generality, we may let $i=1$  so that $S_1 = \bfs{1}$ and $S_2 = \bfs{1}$ or $S_2 = \bfsb{1}$, the former giving that $S_1'=\Szb$ or $S_1'=\Stb$. In any case, $S_1'+S_2=\Z_5$. Thus $G'$ is unrestricted and so it is not $X$-forcing.
\end{proof}

Note that as the forcing set $S$ of a series-parallel graph $G$ is unique, saying it is in $\cup \ssT$ means that it is in $\ccT_S$.

\begin{lemma}\label{lem:ccF-claim2} 
Every minimally $S$-forcing series-parallel graph is in $\cup \ssT$.  
\end{lemma}

\begin{proof} 
Again the proof is by induction on the number of  edges of the minimally forcing series-parallel graph $G$. Note that by the minimality of $G$, $G$ is connected.
The statement that $G$ is in $\cup \ssT$ is true when  $|E(G)| = 1$, as then $G$ is $K_2$. 
Let $G$ be a minimally forcing series-parallel graph with at least two edges,  and with forced set $X$. Note that $X\neq\emptyset$ and $X \neq \Z_5$.  
As $G$ is not $K_2$, $G$ is the serial or  parallel sum of series-parallel graphs $G_1$ and $G_2$, which have forced sets $S_1$ and $S_2$, respectively. 
If $S_1=\Z_5$, then $X = S_1+S_2=\Z_5$ or $X = S_1\cap S_2=S_2$, which implies that $G$ is unrestricted or $G$ is not minimally $X$-forcing.
Thus both $S_1$ and $S_2$ are proper subsets of $\Z_5$ and $G_1$ and $G_2$ are restricted.
From the minimality of $G$,  $G_1$ and $G_2$ are minimally $S_1$-forcing and minimally $S_2$-forcing, respectively.
By the induction hypothesis, $G_1,G_2\in \cup \ssT$, and so $G_1 \in \ccT_{S_1}$ and $G_2 \in \ccT_{S_2}$. 

Suppose that $G=G_1\psum G_2$. Then $X=S_1\cap S_2$. If $S_1\subset S_2$ or $S_2\subset S_1$, then $G$ would not be minimal,  so $S_1$ and $S_2$ must both be in $\{\Szb, \Sob, \Stb\}$. Referring to \Cref{def:families:ccF}, we thus get that $G$ is in $\cup \ssT$, as needed.

Suppose that $G=G_1+G_2$. Then $X=S_1+S_2$.
Again using \Cref{def:families:ccF}, if $S_1$ or $S_2$ is $\Sz$, then $G\in \cup \ssT$.
Suppose that $S_1,S_2\neq \Sz$. 
If either $S_1 \subsetneq S_2$ or $S_1, 
S_2 \in \{\Szb, \Sob, \Stb\}$, then $X = S_1+S_2=\Z_5$ (by \cref{table:C5}), which is a contradiction. 
Thus either $S_1$ and $S_2$ are incomparable or $S_1=S_2=\bfs{i}$ for some $i\in\{1,2\}$. That is, $S_1=\bfs{i}$ and $S_2 \in \{\bfs{i}, \bfs{3-i}, \bfsb{i}\}$ for some $i \in \{1,2\}$. 
Without loss of generality, we may let $i=1$.
Referring to \Cref{def:families:ccF}, $G$ is in $\cup \ssT$, as needed, unless we have $G_1 \in \To$ and $G_2 \in \Tt$.   
We show that this cannot happen, by showing that such $G = G_1 + G_2$, which would be $\Szb$-forcing,  is not minimally forcing.  
Indeed, by induction, $G_2 \in \Tt$ can be written as  
\[ Z_0 + G_2' + Z_1 \ \text{ or } \ Z_0+G_2' \ \text{ or } \ G_2' + Z_1 \ \text{ or } \ G_2'\]
for  {$Z_0, Z_1 \in \Tz $} and $G_2' \in \Tzb \setpsum \Tob$, so $G_2$ can be written as  
\[ Z_0 + (G_3 \psum G_4) + Z_1 \ \text{ or } \ Z_0+(G_3 \psum G_4)\  \text{ or } \ (G_3 \psum G_4) + Z_1 \ \text{ or } \ G_3 \psum G_4 \]
for {$G_3 \in \Tzb$ and $G_4 \in \Tob$}.  
But then $G = G_1 + G_2$ contains the proper subgraph $G'$, which can be written as 
{\[G_1 + (Z_0 + G_4+Z_1) \ \text{ or }\ G_1 + (Z_0 + G_4)\ \text{ or }\ G_1 + (G_4+Z_1)\ \text{ or }\ G_1 +G_4\]}
In each case, $G'$ is in $\To \setssum \Tob \subseteq \Tzb$ with the same terminals 
as $G$.
This contradicts the fact that $G$ is a minimally forcing series-parallel graph.
\end{proof} 
These two lemmas yield the theorem. 

\subsection{Proof of~\Cref{thm:main} and~\Cref{cor:twoC5s}}
 
Now we can prove the main theorem. 

\begin{proof}[Proof of~\Cref{thm:main}]
Take a $C_5$-critical series-parallel graph $G$.
If $G$ is a complete graph, then since $G$ is a series-parallel graph that is not $C_5$-colourable, $G = K_3$.
We know that $K_3 =P_2 \psum P_3$ where $P_2 \in \Fo$ and $P_3 \in \Fob$.
(Here, $P_i$ denotes the path graph on $i$ vertices.)
Thus, from now on, we may assume that $G$ is not a complete graph.
By~\Cref{lem:minimal_2k+1} (i), $G$ is 2-connected and so $G$ has a vertex cut $\{\us,\vt\}$.
So $G=G_1\psum G_2$ for some $(\us,\vt)$-terminal series-parallel graphs $G_1$ and $G_2$. By the criticality of $G$, $G_1$ and $G_2$ are $C_5$-colourable. Let $S_1$ and $S_2$ be the forced sets of  {$G_1$ and $G_2$}, respectively. 
Note that $S_1,S_2\neq\emptyset$.
By~\Cref{lem:minimal_2k+1} (ii), $G_1$ and $G_2$ are restricted, and so $S_1,S_2$ are proper subsets of $\Z_5$.
Since $G$ is not $C_5$-colourable, 
$$S_1\cap S_2=\emptyset.$$ 
Moreover, by the criticality of $G$,  $G_1$ and $G_2$ are minimally $S_1$-forcing and minimally $S_2$-forcing, respectively.
That is, $G_1\in \ccF_{S_1}$ and $G_2\in \ccF_{S_2}$.
Towards a contradiction to the statement of the theorem, suppose that $S_1,S_2 \in \{\bfs{0},\bfs{1},\bfs{2}\}$.
We may assume that $G_1\not\in \Fo$. Since $G_1\neq K_2$, from~\Cref{def:families:ccF}, it follows that $G_1$ can be written as
\[Z_0+(G_3\psum G_4)+Z_1\ \text{ or }\ Z_0+(G_3\psum G_4)\ \text{ or }\ (G_3\psum G_4)+Z_1\  \text{ or }\ G_3\psum G_4\]
for  {$Z_0,Z_1\in \Fz$}, $ {G_3}\in \ccF_{\bfsb{i}}$, and $ {G_4}\in \ccF_{\bfsb{j}}$,  where $i\neq j$  and $i, j \in \{ 0,1,2\}$.  
Since $S_2$ is disjoint from one of $\bfsb{i}$ and $\bfsb{j}$,  
either the following graphs are not $C_5$-colourable
\[(Z_0+G_3+Z_1)\psum G_2, (Z_0+G_3)\psum G_2, (G_3+Z_1)\psum G_2, G_3\psum G_2\]
or the following graphs are not $C_5$-colourable,
\[(Z_0+G_4+Z_1)\psum G_2, (Z_0+G_4)\psum G_2, (G_4+Z_1)\psum G_2, G_4\psum G_2,\]
which is a contradiction to the criticality of $G$.
The assumption that $S_1,S_2\in \{\Sz,\So,\St\}$ is therefore not true, and so $G$ is in $\ccF_{\mathbf{s}(i)} \setpsum \ccF_{\bar{\mathbf{s}}(i)}$ for some $i\in\Z_5$.

Conversely, suppose that $G$ is a graph in the family $\ccFF{i} \setpsum \ccFFb{i}$ for some $i \in \Z_5$.
Clearly, $G$ is not $C_5$-colourable by definition.
Since $G$ is the parallel sum of two graphs in $\cup \ssF$, let $G=G_1\psum G_2$, where $S_1$ and $S_2$ are the forced sets of  $G_1$ and $G_2$, respectively.  
Take any edge $e$ of $G$, and without loss of generality assume $e$ belongs to $G_1$. By~\Cref{lem:families:ccF},  {$(G_1,\us,\vt)$} is minimally $S_1$-forcing and so  {$(G_1- e,\us,\vt)$} is $S_1'$-forcing for some proper superset $S_1'$ of $S_1$. 
This implies $S_1'\cap S_2 \neq \emptyset$ by the cardinality of $S_1$ and $S_2$. Thus $G - e$ is $C_5$-colourable. This completes the proof.
\end{proof}

\begin{lemma}\label{lem:F}
For a $2$-terminal series-parallel graph $G=(G,\us,\vt)$, the following hold:
\begin{itemize}
\item[(i)] If $G$ is in $\cup \ssF$ and has more than five vertices, then $G$ has two $5$-cycles sharing a vertex. 
\item[(ii)] If $G$ is in $\cup \ssF$ and has more than nine vertices and $G\neq H_4 + H_5$  {(recall the base graphs $H_i$'s)}, then $G$ has an $8$-cycle or a $10$-cycle.
\end{itemize}  
\end{lemma}

\begin{proof}
(i) We proceed by induction on the number of vertices. The smallest graph in $\cup \ssF$ on at least 6 vertices  is the base graph $H_5$,
and it has two $C_5$ sharing one vertex.
Suppose that (i) is true for all graphs in $\cup \ssF$ on at most $n-1$ vertices for some $n\ge7$. 
Let $G$ be a graph in $\cup \ssF$ on $n$ vertices.
By~\Cref{def:families:ccF}, $G$ is the serial  or  parallel sum of two graphs $G_1$ and $G_2$ in $\cup \ssF$. If $|V(G_i)|\ge 6$, then $G_i$ has two $C_5$ sharing a vertex by the induction hypothesis.
Thus, each $G_i$ has at most five vertices, and so it is one of the base graphs $H_1$, $H_2$, $H_3$, $H_4$.
Thus, $|V(G)|\le 9$. 
Since $7 \le |V(G)|\le 9$, $G = H_5$. Thus, $G$ has two $5$-cycles sharing a vertex.
  
(ii) Again we use induction on the number of vertices. The only graph in $\cup \ssF$ with 10 vertices {is $H_6$, and it contains} an $8$-cycle. 
Suppose that (ii) is true for all graphs in $\cup \ssF$ on at most $n-1$ vertices for some $n\ge 11$. 
Take a graph $G$ in $\cup \ssF$ on $n$ vertices. By~\Cref{def:families:ccF}, $G$ is the serial  or  parallel sum of graphs $G_1$ and $G_2$ {in $\cup \ssF$}. 

Suppose that $G_1=H_4+H_5 \in \Fzb$. 
If $G=G_1 + G_2$, then by~\Cref{def:families:ccF}, $G_2 \in \Fz$ so that $|V(G_2)| \ge 10$ and $G_2 \neq H_4 + H_5$.
By the induction hypothesis, $G_2$ has an $8$-cycle or a $10$-cycle, which implies that $G$ has an $8$-cycle or a $10$-cycle. 
If $G= G_1 \psum G_2$, then by~\Cref{def:families:ccF}, $G_2\in \Fob \cup \Ftb$ so that $G_2\neq H_4+H_5$. 
If $|V(G_2)|\ge 10$, then we are done by the induction hypothesis.
If $|V(G_2)|\le 9$, then $G=H_2 \psum (H_4+H_5 )$
or $G=H_5 \psum (H_4+H_5)$ so that $G$ contains an $8$-cycle or a $10$-cycle. 
See~\Cref{fig:c8c10} for illustrations.
  
\begin{figure}[h!]
    \centering
    \includegraphics[width=10cm]{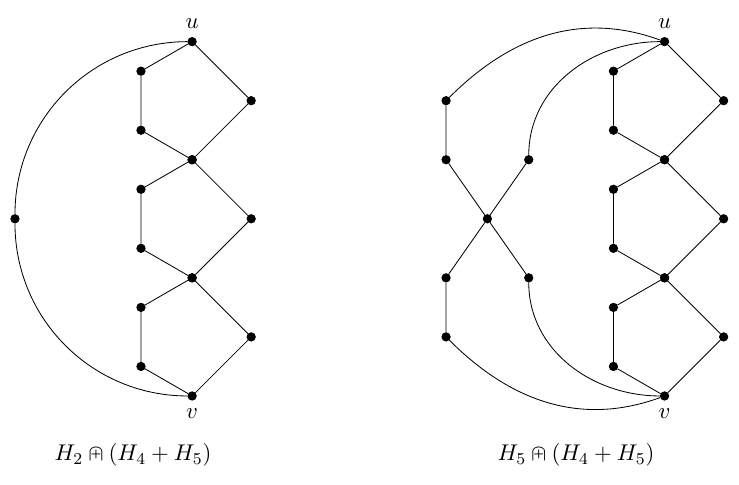}
    \caption{Illustrations for $H_2 \psum (H_4+H_5)$ and $H_5 \psum (H_4 + H_5)$.}
    \label{fig:c8c10}
\end{figure}
\begin{figure}[h!]
    \centering
    \includegraphics[width=8cm]{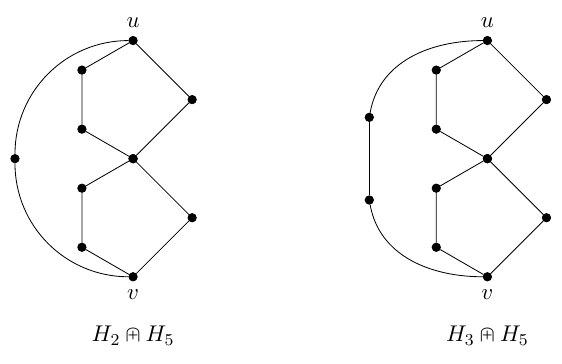}
    \caption{Illustrations for $H_2 \psum H_5$ and $H_3 \psum H_5$}
    \label{fig:c8c10_2}
\end{figure}

{Suppose that $G_i \neq H_4 + H_5$ for all $i \in [2]$.}
By the induction hypothesis, we may assume that each $G_i$ has at most nine vertices, and so it is one of the base graphs $H_1$, $H_2$, $H_3$, $H_4$, $H_5$.
Since  each of $H_1$, $H_2$, $H_3$, $H_4$ has at most five vertices, 
$G_1$ or $G_2$ is $H_5\in \Ftb$.
By~\Cref{def:families:ccF} we get that $G$ is one of $H_4 + H_5$, $H_2 \psum H_5$, $H_3 \psum H_5$.
Since $G\neq H_4 + H_5$, $G$ is 
$H_2 \psum H_5$ or $H_3 \psum H_5$.
See~\Cref{fig:c8c10_2} for illustrations.  
Thus $G$ has an $8$-cycle or a $10$-cycle.
\end{proof}

\begin{proof}[Proof of~\Cref{cor:twoC5s}]
For each statement, we let $G$ be a minimum counterexample. 
So $G$ is a $C_5$-critical series-parallel graph. By~\Cref{thm:main}, $G$ is a graph in the family $\ccF_{\mathbf{s}(i)} \setpsum \ccF_{\bar{\mathbf{s}}(i)}$ for some $i\in\{0,1,2\}$.

For statement (i), we assume that $G$ has no two $5$-cycles sharing a vertex.
By~\Cref{lem:F}~(i), every graph in $\cup \ssF$  with more than five vertices has two $5$-cycles sharing a vertex. 
So $G$ is a parallel sum of two graphs among the four base graphs $H_1$, $H_2$, $H_3$, $H_4$. Since $G$ is in $\ccF_{\bfs{i}} \setpsum \ccF_{\bfsb{i}}$  for some $i\in\{0,1,2\}$,  $G$ is  $H_1 \psum H_2$. Then $G=K_3$, which is a contradiction.

For statement (ii), we assume that $G$ has neither $C_8$ nor $C_{10}$. 
By~\Cref{lem:F} (ii), we conclude that either $G$ has at most nine vertices or $G=H_4+H_5$.
Clearly, $G\neq H_4+H_5$, since $H_4+H_5$ is $C_5$-colourable.
Thus $|V(G)|\le 9$.
So $G$ is in $\Fo \setpsum \Fob$, and so we have $G=H_1\psum H_2$,  which means $G = K_3$, a contradiction.
\end{proof}

\section{Generalisation to $(2k+1)$-cycles for $k \geq 2$}

In this section, we prove an extension of our main technical theorem,  \Cref{thm:main}, to larger odd cycles.   This requires more definitions. Throughout this section, let $k \ge 2$ be an integer.
Refining the set $\ccF_S$ of minimally $S$-forcing series-parallel graphs, we make the following definition.  

\begin{definition}\label{def:forcingfam}
For $S\subseteq \Z_{2k+1}$, let $\bar{S}=\Z_{2k+1} \setminus S$.
For disjoint nonempty symmetric subsets $S,T \subseteq \Z_{2k+1}$, let $\ccF_S^T$ be the family of minimally $S$-forcing $(\us,\vt)$-terminal series-parallel graphs $(G,\us,\vt)$ such that
for every $e \in E(G)$ there is an element $x \in T$ and a $C_{2k+1}$-colouring $\phi$ of $G - e$ such that $\phi(\us) = 0$ and $\phi(\vt) = x$. 
\end{definition}

Note that if $T\subseteq T'$, then $\ccF_S^{T}\subseteq  \ccF_S^{T'}$, and so $\ccF_S: = \ccF_S^{\bar{S}}$ is exactly the family of minimally $S$-forcing series-parallel graphs.

Our main task in this section is to prove the following theorem. 

\begin{theorem}\label{cor:minimal_2k+1} 
A series-parallel graph $G$ is $C_{2k+1}$-critical if and only if $G$ belongs to $\ccF_S^A \setpsum \ccF_T^B$ for some nonempty symmetric subsets $S,T,A,B$ of $\Z_{2k+1}$ such that  $S\cap T = \emptyset$, $A\subseteq T$, and $B\subseteq S$.
\end{theorem}

\Cref{cor:minimal_2k+1} can be viewed as a generalisation of \Cref{thm:main}.
Indeed, before we prove \Cref{cor:minimal_2k+1}
 we show how \Cref{thm:main} follows from it as a corollary.

 We first observe that some of the families $\ccF_S^T$ are empty for cardinality reasons.

\begin{proposition}\label{prop}
   The family $\ccF_\bfs{i}^\bfs{j}$ is empty unless $i = 1$, in which case it contains only $K_2$ and serial sums of $K_2$ and elements of $\ccF_\bfs{0}$. 
\end{proposition}
\begin{proof}
    Observe that no family $\ccF_{\bfs{i}}^{\bfs{j}}$ can contain the parallel sum $(G,\us,\vt) =  G_1 \psum G_2$. Indeed, let $S$ and $T$
    be the forced sets of $G_1$ and $G_2$, respectively. 
    Both properly contain the set $S \cap T$ that is forced by $G$.
    Now, for every edge $e$ in {$G_1$}, $G_1 - e$ admits a $C_5$-colouring $\phi$ with $\phi(\us) = 0$ and $\phi(\vt)$ being in $T \setminus S$. So for $G$ to be in $\ccF_{\bfs{i}}^{\bfs{j}}$, $\bfs{j}$ must contain some element of $T \setminus S$. Similarly, it must also contain some element of $S \setminus T$.  But these are symmetric sets, and $\bfs{j} = \{\pm j\}$, so this is impossible. 

    It is therefore sufficient to prove the proposition when $G \in \ccF_\bfs{i}^\bfs{j}$ is $K_2$ or the serial sum $G_1 + G_2$.
    In the latter, letting $S$ and $T$ be the forced sets of $G_1$ and $G_2$, respectively, the fact that $S + T = \bfs{i}$ clearly implies that $\{ S,T \} = \{ \bfs{0}, \bfs{i} \}$, so we may assume, without loss of generality, that $S = \bfs{0}$ and $T = \bfs{i}$.  

    In the case that $i \neq 1$, we assume towards contradiction, that the serial sum $G = (G_1,\us_1,\vt_1) + (G_2,\us_2,\vt_2)$ is a smallest graph in $\ccF_{\bfs{i}}^{\bfs{j}}$.    
    For every edge $e$ of $G_2$ we have that $G - e$ has a $C_5$-colouring $\phi$ with $\phi(\us_1) = \phi(\us_2) = 0$ and {$\phi(\vt_2) \in \bfs{j}$}.  So $G_2$ is in $\ccF_{\bfs{i}}^{\bfs{j}}$, contradicting the minimality of $G$. 

    The proof in the case that $i = 1$ is the same, only we assume that $G = G_1 + G_2$ is a smallest graph in $\ccF_{\bfs{i}}^{\bfs{j}}$ that is not the {serial} sum of $K_2$ and a graph from $\ccF_\bfs{0}$. 
\end{proof}

Now \Cref{thm:main} follows from \Cref{cor:minimal_2k+1} by  Proposition~\ref{prop}. Indeed, by \Cref{cor:minimal_2k+1}, a series-parallel $C_5$-critical graph $G$ is in $\ccF_S^A \psum \ccF_T^B$ where $S,T,A$, and $B$ are symmetric subsets of $\Z_5$ with $S \cap T = \emptyset, A \subset T$, and $B \subset S$. 
As $S$ and $T$ are disjoint, we may assume that $S = \bfs{i}$ for some $i$, and so its nonempty subset $B$ is also $\bfs{i}$. As $T$ is disjoint from $S = \bfs{i}$, it is either $\bfsb{i}$ or $\bfs{j}$ for some $j \neq i$. If it is $\bfs{j}$, then $A \subset T$ is also $\bfs{j}$, and by \Cref{prop}, one of $\ccF_S^A$ and $\ccF_T^B$ is empty.  So $T = \bfsb{i}$, giving \Cref{thm:main}.

With this, the original proof of \Cref{thm:main} is perhaps unnecessary. However, \Cref{thm:main} was presented first to highlight the simplicity of the statement regarding $C_5$-colourings, and its proof, hopefully, gives the reader intuition for the following more complicated proof of \Cref{cor:minimal_2k+1}.

\begin{definition}      
For two nonempty subsets $S$ and $Q$ of $\Z_{2k+1}$, let $Q_S:= \{ x \mid (x+S) \cap Q \neq \emptyset\}$ be the set of elements $x$ such that $x + s$ is in $Q$ for some $s \in S$.  
\end{definition}

For example,  in $C_7$,  $(\So\cup \Sth)_{\So} = \Sob$ and $(\Sz\cup\St)_{\So} = \So\cup\Sth$.  
Note that $Q_S$ is symmetric if $Q$ and $S$ are symmetric.

\begin{fact}\label{fact:Qs}
For two subsets $S$ and $Q$ of $\Z_{2k+1}$, we have $ 
\overline{Q_S} \subseteq \overline{Q}_S$; indeed
\[\begin{array}{ll}
 x \in \overline{Q_S} & \Longrightarrow  x \not\in Q_S \\
    &\Longrightarrow x+s \not\in Q \text{ for every }s\in S\\
    &\Longrightarrow x+s \in \overline{Q} \text{ for every }s\in S\\ 
    &\Longrightarrow x+s \in \overline{Q} \text{ for some }s\in S\\
    &\Longrightarrow x \in \overline{Q}_S.
\end{array}\]
\end{fact}

The following is an extension/refinement of \Cref{lem:families:ccF}. 

\begin{lemma}\label{lem:tech}
 For disjoint nonempty symmetric subsets $Q, R \subseteq \Z_{2k+1}$, the graphs in $\ccF_{R}^{Q}$
are those that are constructed by putting $K_2$ into $\ccF_{\bfs{1}}^X$ for every symmetric subset $X \subseteq \bfsb{1}$ 
and defining families recursively by the following constructions.
\begin{enumerate}
    \item[(1)] For every symmetric sets $S$ and $T$ such that $R= S \cap T$, $Q \cap (T \setminus S) \neq \emptyset$, and $Q \cap (S \setminus T) \neq \emptyset$, if $G_1 \in \ccF_{S}^{Q \cap (T \setminus S)}$ and $G_2 \in \ccF_{T}^{Q \cap (S \setminus T)}$, then $G_1 \psum G_2 \in \ccF_R^Q$.
    \item[(2)] For every symmetric sets $S$ and $T$ such that $R= S + T$, $Q_T \neq \emptyset$, and $Q_S \neq \emptyset$, if $G_1 \in \ccF_{S}^{Q_T}$ and $G_2 \in \ccF_{T}^{Q_S}$, then $G_1 + G_2 \in \ccF_R^Q$.
\end{enumerate}
\end{lemma}
   
\begin{proof}  The proof consists of two claims that verify simple properties of our definitions. 
\begin{claim}
  The minimal parallel sums in $\ccF_R^Q$ are exactly those in 
             \[\ccF_S^{ Q \cap (T\setminus S)} \setpsum \ccF_T^{Q \cap (S\setminus T)} \]
  for symmetric sets $S$ and $T$ with $R = S \cap T$.  (Clearly if $Q \cap (T\setminus S)$ or $Q \cap (S\setminus T)$ are empty, then so is this family.) 
  \end{claim}
  \begin{proof}\claimproof
      Let $G_1 \psum G_2 \in \ccF_R^Q$.  
      Since $\ccF_R^Q$ consists of minimally $R$-forcing graphs, the forced sets of $G_1$
      and $G_2$ are some $S$ and $T$, respectively, with $S \cap T = R$.  
      On the removal of an edge $e$ of $G_1$, the forced set of $G - e$ contains a new element of $Q$, so that of $G_1 - e$ does too. This newly forced element is in $T\setminus S$, since otherwise it would either not be in the forced set of $G - e$, or it would already have been in the forced set of $G$.  
      So it is in $Q \cap (T \setminus S)$. 
      Thus $G_1 \in \ccF_S^{Q \cap (T \setminus S)}$, as needed.
      By the same argument, $G_2 \in \ccF_T^{Q \cap (S \setminus T)}$.

      On the other hand, assume that for $G = G_1 \psum G_2$ we have $G_1 \in \ccF_S^{ Q \cap (T\setminus S)}$ and $G_2 \in \ccF_T^{Q \cap (S\setminus T)}$.  
      Clearly, the forced set of $G$ is $R = S \cap T$. Moreover, on the removal of an edge $e$ of $G_1$ from $G$, the forced set of $G_1 - e$ contains some new element in $Q \cap (T \setminus S)$ so that it is in $Q$ and already in the forced set of $G_2$. So this new element of $Q$ is in the forced set of {$G-e$}.  The same holds when switching $G_1$ and $G_2$.  So on the removal of every edge $e$, the forced set of $G - e$ contains some new element of $Q$. That is, $G \in \ccF_R^Q$, as needed. 
  \end{proof}
  
The following claim is a generalisation of the statement that for a set $R$, the minimal serial sums in $\ccF_R$ are those in 
         \[{\ccF_{S}^{\overline{(S+T)}_T} \setssum \ccF_{T}^{\overline{(S+T)}_S}}\]
      for every symmetirc sets $S$ and $T$ for which $S + T = R$.   Intuitively: 
      when we remove an edge from the factor in 
       $\ccF_{T}^{\overline{(S+T)}_S}$ we want to allow a new colour in its forced set that, when we sum with $S$, gives a colour that was not in $(S + T)$.     

\begin{claim}
The minimal serial sums in $\ccF^Q_R$ (for disjoint $R$ and $Q$)  are  exactly those in 
       \[ \ccF_{S}^{Q_T} \setssum \ccF_{T}^{Q_S} \]
for symmetric sets $S$ and $T$ for which $S + T = R$.       \end{claim}
\begin{proof} \claimproof
Let $G_1 + G_2 = G \in \ccF_R^Q$. Since $\ccF_R^Q$ consists of minimally $R$-forcing graphs, we may assume $G_1$ and $G_2$ force sets $S$ and $T$, respectively, with $R=S+T$.
    On the removal of an edge $e$ from $G_1$, $G- e$ forces a new element of $Q$, so $G_1- e$ forces some new element $x$   such that $x + t \in Q$ for some $t$ in $T$; thus $x$ is in $Q_T$, and so $G_1 \in \ccF_S^{Q_T}$, as needed. Note that as $Q$ is disjoint from $R$, and $S + T = R$, nothing in $S$ is in $Q_T$, so $S$ and $Q_T$ are disjoint, making $\ccF_S^{Q_T}$ well-defined. By the analogous argument, $G_2 \in \ccF_T^{Q_S}$.

On the other hand, let $G = G_1 + G_2$  with $G_1 \in \ccF_S^{ Q_T}$ and $G_2 \in \ccF_T^{Q_S}$ where $S + T = R$.  We have that $G_1$ and $G_2$ force the sets $S$ and $T$, respectively, so $G$ is $R$-forcing. Removing an edge $e$ of $G_1$ from $G$, we get that $G_1- e$ forces some $x \in Q_T$.  As $G_2$ forces $T$, $G-e = (G_1-e)+G_2$ forces {$x+t \in Q$ for some $t \in T$.}  Thus $G - e$ forces some element of $Q$, as needed. 
\end{proof}
\end{proof}

\begin{proof}[Proof of~\Cref{cor:minimal_2k+1}]
Let $G$ be a $C_{2k+1}$-critical series-parallel graph.
By~\Cref{lem:minimal_2k+1}~(i), $G$ is $2$-connected.
This implies that $G = G_1 \psum G_2$ for some $(\us,\vt)$-terminal series-parallel graphs $G_1$ and $G_2$.
By the criticality of $G$, $G_1$ and $G_2$ are $C_{2k+1}$-colourable.
Let $S \neq \emptyset$ and $T \neq \emptyset$ be the forced sets of $G_1$ and $G_2$, respectively.
Since $G$ is not $C_{2k+1}$-colourable, $S$ and $T$ are disjoint, so $S$ and $T$ are proper symmetric subsets of $\Z_{2k+1}$.
By the criticality of $G$, $G_1$ and $G_2$ are minimally $S$-forcing and minimally $T$-forcing, respectively.
Thus, $G_1 \in \ccF_S$ and $G_2 \in \ccF_T$ for some disjoint $S$ and $T$.
For every edge $e$ of $G_1$, the forced set of $G_1 - e$ contains an element in $T$, so $G_1 \in \ccF_S^T$. Similarly we get that $G_2 \in \ccF_T^S$,
which satisfies the conclusion of the theorem with $A = T$ and $B = S$. 
Thus we may let $A \subseteq T$ {and $B \subseteq S$}.

For the converse, let $G = G_1 \psum G_2$ where $G_1\in \ccF_S^A$, $G_2 \in \ccF_T^B$ for some nonempty symmetric subsets $S, T, A, B$ of $\Z_{2k+1}$ such that $S \cap T = \emptyset$, $A \subseteq T$, and $B \subseteq S$.
Since $S \cap T = \emptyset$, $G$ is not $C_{2k+1}$-colourable.
Let $e$ be an edge of $G$, and, without loss of generality, assume that it is an edge of $G_1$.
Since $G_1 \in \ccF_S^A$, $G_1-e$ forces a new element in $A \subseteq T$.
Thus, $G-e = (G_1-e) \psum G_2$ forces a new element in $T$, and thus, $G-e$ is $C_{2k+1}$-colourable.
Therefore, $G$ is a $C_{2k+1}$-critical series-parallel graph.
\end{proof}

\section*{Acknowledgements}
This work was started during the 2023 Winter Workshop in Combinatorics.  Eun-Kyung Cho was supported by Basic Science Research Program through the National Research Foundation of Korea(NRF) funded by the Ministry of Education (No. RS-2023-00244543).
Ilkyoo Choi was supported by the Institute for Basic Science (IBS-R029-C1) and the Hankuk University of Foreign Studies Research Fund.
Boram Park was supported by the National Research Foundation of Korea(NRF) grant funded by the Korea government (MSIT) (No. RS-2025-00523206).
Mark Siggers was  supported by the Basic Science Research Program through the National Research Foundation of Korea (NRF-2022R1A2C1091566) and the Kyungpook National University Research Fund.

\end{document}